\documentclass[a4paper,11pt]{amsart}
\usepackage{amsmath,amssymb,amsthm,amscd,xypic}
\usepackage{enumerate}
\usepackage{bbold}

\title{Reasonable triangulated categories have filtered enhancements}

\author{George Ciprian Modoi}
\address{Babe\c s--Bolyai University, Faculty of Mathematics and Computer Science \\  1, Mihail Kog\u alniceanu, 400084 Cluj--Napoca, Romania}
\email{cmodoi@math.ubbcluj.ro}

\thanks{}


\date{\today}

\newcommand{\N}{\mathbb{N}}
\newcommand{\Z}{\mathbb{Z}}

\DeclareMathOperator{\Hom}{Hom}

\DeclareMathOperator{\Ext}{Ext} 

\DeclareMathOperator{\Ker}{Ker}
\DeclareMathOperator{\Img}{Im}

\DeclareMathOperator{\hocolim}{{\mathrm{hocolim}}}
\DeclareMathOperator{\holim}{{\mathrm{holim}}}

\DeclareMathOperator{\Ho}{H}

\newcommand{\A}{\mathcal{A}}

\newcommand{\C}{\mathcal{C}}
\newcommand{\D}{\mathcal{D}}

\newcommand{\CS}{\mathcal{S}}
\newcommand{\T}{\mathcal{T}}

\newcommand{\U}{\mathcal{U}}
\newcommand{\X}{\mathcal{X}}

\newcommand{\n}{\mathbb{n}}
\newcommand{\0}{\mathbb{0}}
\newcommand{\one}{\mathbb{1}}
\newcommand{\two}{\mathbb{2}}

\newcommand{\Cat}{\hbox{\rm Cat}}
\newcommand{\CAT}{\hbox{\rm CAT}}
\newcommand{\Ab}{\mathcal{A}b}

\newcommand{\opp}{^\textit{op}}
\newcommand{\id}{\mathrm{id}}

\newcommand{\Der}[1]{\mathsf{D}({#1})}
\newcommand{\DER}{\mathsf{D}}
\newcommand{\dia}{\mathsf{dia}}
\newcommand{\real}{\mathsf{real}}
\newcommand{\Sp}{\mathcal{SP}}
\newcommand{\HoSp}{\mathcal{SH}}

\theoremstyle{plain}
\newtheorem{thm}{Theorem}
\newtheorem{lem}[thm]{Lemma}
\newtheorem{prop}[thm]{Proposition}
\newtheorem{cor}[thm]{Corollary}

\theoremstyle{definition}

\theoremstyle{remark}
\newtheorem{rem}[thm]{Remark}
\newtheorem{expl}[thm]{Example}

\begin{document}
\hyphenation{-ap-prox-i-ma-tions}

\begin{abstract}
We prove that a triangulated category which is the underlying category of a stable derivator has a filtered enhancement, providing an affirmative answer 
to  a conjecture in \cite{Bo09}. 
\end{abstract}

\maketitle


\section*{Introduction}
Filtered enhancements of triangulated categories were defined in order to allow us to perform various constructions in these categories:
In \cite[Appendix]{Be87}, where filtered enhancement  were first considered, they are used for defining the realization functor. This is a functor 
$\D^b(\A)\to\T$, where $\T$ is the enhanced triangulated category and $\D^b(\A)$ is the bounded derived category over the abelian category $\A$, 
which occurs as the heart of an appropriate 
t-structure in $\T$. In \cite{PV17} the realization functor is used in the study of tilting and silting objects in triangulated categories 
(see for example \cite[Proposition 5.1]{PV17}). 
In  \cite{Bo09} and \cite{Sch11} filtered enhancement are used for studying the so called the weight complex functor of \cite[Theorem 3.3.1]{Bo09}. 
For example in \cite[Conjecture 3.3.3 and Remark 3.3.4]{Bo09} it is stated that the weight complex functor can be lifted to a strong version, provided that a filtered 
enhancement exists.
In \cite[Remark 3.3.4]{Bo09} it is conjectured that every reasonable triangulated category has a filtered enhancement; in a private communication, M. Bondarko credited Be\u\i linson 
with the formulation of this conjecture. 
As an evidence for truth of this conjecture, remark that in \cite[Example A 2]{Be87} it is stated that the filtered derived category in the sense of \cite[V.1]{Il71} 
provides a filtered enhancement for the usual derived category $\D(\A)$ of an abelian category $\A$. 
The whole argument for this fact can be found in \cite[Proposition 6.3]{Sch11}. 
The present paper originates in the observation that the same argument works for showing that the derived category of modules over a dg algebra has a filtered enhancement.  
Together with \cite[Proposition 3.8]{PV17}, which says that filtered enhancements are inherited from a triangulated category to an arbitrary localization, 
and with Keller and Porta characterization of well generated algebraic triangulated categories as being localizations of derived categories of modules over dg algebras 
(see \cite{Ke94} and \cite{Po10}) the above conjecture is settled for this important class of triangulated categories. Further, the idea is to generalize this approach 
by mimicking the argument above in order to construct filtered enhancement for the underlying category $\Der\0$ of a stable derivator $\DER$ inside the category $\Der\Z$, 
where $\Z$ is the poset of integers. This is what we are doing in this work.  We end the paper by applying our abstract results in order to find a new proof for a 
well--known result from  algebraic topology. 

Finally it is worth to mention a new work which comes to complete our approach: in the preprint \cite{V18} it is shown that if we work in the 
underlying category of a stable derivator, then the realization functor can be defined at
the level of the unbounded derived category. Among others, the preprint \cite{V18} contains the same intuition as this paper, that the f-enhancement can be constructed at the level of $\Der\Z$, but, as it will demonstrate in the following,  the proof of this fact is not quite obvious.    

{\sc Acknowledgements.} The author would like to express many thanks to Mikhail Bondarko, Olaf Schn\"urer and Jorge Vit\'oria for stimulated 
discussions concerning the subject of this work.  He is also very indebted to an anonymous referee for many suggestions which lead to a 
substantial improvement of this paper. 

\section*{The construction of f-enhancements}

We begin with recalling the definitions of the main notions we will use in the sequel.   
Consider a triangulated category $\T$, whose suspension functor is denoted by $X\mapsto\Sigma X$, for all $X\in\T$. Then, according to \cite[Definition 1.3.1]{BBD82}, a 
{\em t-structure} 
in $\T$ is a pair of full subcategories $\left(\T^{\leq0},\T^{\geq0}\right)$, such that the following conditions are satisfied:
\begin{enumerate}[{\rm (T1)}]
\item $\T^{\leq0}\subseteq\T^{\leq1}$ and $\T^{\geq1}\subseteq\T^{\geq0}$, where $\T^{\geq n}=\Sigma^{-n}\T^{\geq0}$ and $\T^{\leq n}=\Sigma^{-n}\T^{\leq0}$, for all $n\in\Z$.
 \item $\T(X,Y)=0$ for all $X\in\T^{\leq0}$ and all $Y\in\T^{\geq1}$.
 \item For all $X\in\T$ there is a triangle \[X^{\leq0}\to X\to X^{\geq1}\overset{+}\to\] in $\T$, with $X^{\leq0}\in\T^{\leq0}$ and $X^{\geq1}\in\T^{\geq1}$.
\end{enumerate} 
Note that if $\left(\T^{\leq0},\T^{\geq0}\right)$ is a t-structure on $\T$, then $\T^{\leq0}$ is called the {\em aisle}, $\T^{\geq0}$ is called the 
{\em coaisle} and $\T^{\leq 0}\cap\T^{\geq0}$ is called the {\em heart}, associated with this t-structure. By 
\cite{BBD82} the heart of a t-structure is abelian. We know that the inclusion functor $\T^{\leq0}\to\T$ has a right adjoint and the inclusion functor 
$\T^{\geq0}\to\T$ has a left adjoint. Note also that the t-structure is called {\em stable} if its aisle (or equivalently coaisle) is a triangulated subcategory.

For a functor $F:\T\to\CS$ (between additive categories) we 
consider the essential image $\Img F=\{S\in\CS\mid S\cong F(T)\hbox{ for some }T\in\T\}$ and the kernel $\Ker F=\{T\in\T\mid F(T)\cong 0\}$.  

A sequence of triangulated categories and functors $\CS\underset{r}\rightleftarrows\T\overset{l}\rightleftarrows\U$ 
is called {\em Bousfield localization sequence} if $r$ is the right adjoint of the inclusion functor $\CS\to\T$ 
(here and below, the left adjoint is depicted up and the right one down), $l$ is the left adjoint of the inclusion functor $\U\to\T$ and $\CS=\Ker l$. 
Note that this definition is equivalent to \cite[Definition 9.1.1]{Nee01}. For details, one can consult \cite{K10}. 
Remark also that we have a Bousfield localization sequence as above if and only if 
$(\CS,\U)$ is a stable t-structure in $\T$. 

A {\em filtered triangulated category} or an {\em f-category} for short, is a quintuple $(\X, \X(\geq 0),\X(\leq 0), s,\alpha)$, where $\X$ is a triangulated category, 
$\X(\geq 0)$ and $\X(\leq 0)$ are full triangulated subcategories, $s:\X\to\X$ is a triangle autoequivalence and $\alpha:\id_\X\Rightarrow s$ is a natural transformation; 
we put $\X(\geq n)=s^n\X(\geq0)$ and $\X(\leq n)=s^n\X(\leq 0)$. These data have to satisfy the following axioms:
\begin{enumerate}[{\rm (F1)}]
 \item $\X(\geq1)\subseteq\X(\geq0)$ and $\X(\leq0)\subseteq\X(\leq1)$.
 \item $\X=\bigcup_{n\in\Z}\X(\geq n)=\bigcup_{n\in\Z}\X(\leq n)$. 
 \item $\X(X,Y)=0$ for all $X\in\X(\geq1)$ and all $Y\in\X(\leq0)$.
 \item For all $X\in\X$ there is a triangle \[X(\geq1)\to X\to X(\leq0)\overset{+}\to\] in $\X$, with $X(\geq1)\in\X(\geq1)$ and $X(\leq0)\in\X(\leq0)$.
 \item One has $\alpha\cdot s=s\cdot\alpha$ as natural transformations $s\Rightarrow s^2$.
 \item For all $X\in\X(\geq1)$ and all $Y\in\X(\leq0)$, the map \[\X(s(Y),X)\to\X(Y,X)\hbox{ given by }g\mapsto g\cdot\alpha(Y)\] is bijective. 
\end{enumerate}
In \cite[\S 7.2]{Sch11} it is introduced a new axiom for f-categories, which seems to be necessary in order to show that the realization functor is a triangle functor: 
\begin{enumerate}[{\rm (F7)}]
 \item Let $f:X\to Y$ be a map in $\X$. The morphism of triangles constructed in the solid part of the following diagram fits in the $3\times 3$ diagram 
 whose rows and columns are triangles (such a morphism of triangles is called {\em middling--good} in \cite[Definition 2.4]{Nee91}):
 \[\xymatrix{ X(\geq1)\ar[rr]\ar[d]^{\alpha(Y(\geq1))\cdot f(\geq1)} && X\ar[rr]\ar[d]^{\alpha(Y)\cdot f} && X(\leq0)\ar[r]^{\ \ +}\ar[d]^{\alpha(Y(\leq0))\cdot f(\leq0)}& \\	
 s(Y(\geq1))\ar[rr]\ar@{.>}[d] && s(Y)\ar[rr]\ar@{.>}[d] && s(Y(\leq0))\ar[r]^{\ \ \ +}\ar@{.>}[d]& \\
 Z'\ar@{.>}[d]^{+}\ar@{.>}[rr] && Z\ar@{.>}[d]^{+}\ar@{.>}[rr] && Z''\ar@{.>}[d]^{+}\ar@{.>}[r]^{+} &\\
                               &&                              &&                               &
}\] where by $X\mapsto X(\geq n)$ and $X\mapsto X(\leq n)$ are denoted the right, respectively left adjoint for the inclusions $\X(\geq n)\to\X$, 
respectively $X(\leq n)\to\X$ (the existence of these adjoints is guaranteed by the other axioms of an f-category). 
\end{enumerate}
If $\T$ is a triangulated category then a {\em filtered enhancement} for $\T$, called also an {\em f-enhancement} or an {\em f-category over }$\T$, 
is an f-category $\X$ together with a triangle equivalence \[\T\overset{\sim}\longrightarrow\X(\geq0)\cap\X(\leq0).\] 
 
Derivators were first introduced by Grothendieck in an initially unpublished manuscript which is nowadays available online, thanks to the efforts of 
M. K\"unzer, J. Malgoire and G. Maltsiniotis, see \cite{Groth91}. In our consideration about this subject we will follow closely the exposition in \cite{Gr13}. According to \cite[Definition 1.1]{Gr13},  
a {\em prederivator} is a strict 2-functor $\DER:\Cat\opp\to\CAT$ where $\Cat$ is the 2-category of all small categories and $\CAT$ is the 2-category of 
not necessarily small categories. We neglect set theoretical issues coming from the fact that $\CAT$ has no small hom-sets, since they play no role in our considerations. 
For two small categories $J$ and $K$ and a functor $u:J\to K$ we will denote $u^*=\Der u:\Der K\to\Der J$. We entirely stick on the convention in \cite{Gr13}, in the sense 
that $\DER$ reverses the sense of functors, but preserves the sense of natural transformation. That is, if $\alpha:u\Rightarrow v$ is a natural transformation in $\Cat$, then 
the induced natural transformation is defined as $\alpha^*:u^*\Rightarrow v^*$. With the notations above a left (respectively right) adjoint for $u^*$ is called {\em homotopy left (right) Kan extension along }$u$. We will denote the left adjoint by $u_!:\Der J\to \Der K$ and the right adjoint by $u_*:\Der J\to \Der K$ (provided they exist).  

For every $n\in\N$, denote by $\n$ the category associated to the poset $\{0,1,\ldots,n\}$ with the natural order.  
Therefore $\0=[0]$ is the singleton category, which is a terminal object in $\Cat$, $\one$ is the arrow $[0\to 1]$ etc. Then the category $\Der\0$ is called the 
{\em underlying category} of the prederivator $\DER$. 
Remark that for every small 
category $K$ and every $k\in K$, there is a functor $k:\0\to K$ sending $0$ to $k\in K$. The induced functor $k^*:\Der K\to\Der\0$ is called the {\em evaluation functor}.
For all objects $X,Y$ and all morphisms $f:X\to Y$ in $\Der K$ we will denote $X_k=k^*(X)$ and $f_k=k^*(f)$. The right (left) Kan extension along the unique functor 
$K\to\0$ is called the {\em homotopy (co)limit} of shape $K$ and is denoted $\holim_K$, respectively $\hocolim_K$.  Let $u:J\to K$ be a functor between two small categories. 
For every $k\in K$ we consider the categories 
$J/k$ and $J\backslash k$ whose objects are pairs of the form $(j,u(j)\to k)$, respectively $(j,k\to u(j))$, where $j\in J$ and $u(j)\to k$, $k\to u(j)$ 
are maps in $K$, and whose morphisms are those of $J$ which make commutative the obvious triangles. We call $J/k$ and $J\backslash k$ {\em the categories of   
$u$-objects over}, respectively  {\em under}, $k$.  In both cases, there is a functor $\pi:J/k\to J$, respectively $\pi:J\backslash k\to J$ which forgets the morphism component. 
With these data, one can construct two natural maps (see \cite[Section 1.1]{Gr13} for details): 
\[\hocolim_{J/k}\pi^*(X)\to u_!(X)_k\hbox{ and }u_*(X)_k\to\holim_{J\backslash k}\pi^*(X).\]

A prederivator $\DER:\Cat\opp\to\CAT$ is called {\em derivator} if it satisfies the following axioms:
\begin{enumerate}[{\rm (D1)}]
 \item $\DER$ sends coproducts to products; in particular, $\Der\emptyset$ is trivial.
 \item A morphism $f:X\to Y$ in $\Der K$ is an isomorphism if and only if $f_k:X_k\to Y_k$
is an isomorphism in $\Der\0$ for every $k\in K$. 
 \item For every functor $u:J\to K$, there are homotopy left and right Kan extensions along $u$:
\[\xymatrix{ \Der{J} \ar@/^/[rr]^{u_!}\ar@/_/[rr]_{u_*} &&\Der K\ar[ll]|-{u^*} } .\]
 \item For every functor $u:J\to K$ and every $k\in K$, the canonical morphisms
$\hocolim_{J/k}\pi^*(X)\to u_!(X)_k\hbox{ and }u_*(X)_k\to\holim_{J\backslash k}\pi^*(X)$
are isomorphisms for all $X\in\Der J$.
\end{enumerate}

If $\DER:\Cat\opp\to\Cat$ is a prederivator and $K$ is a small category, one defines $\DER^K:\Cat\opp\to\CAT$, by $\DER^K(J)=\Der{K\times J}$.  
From \cite[Theorem 1.25]{Gr13}, we learn that if $\DER$ is a derivator, then so is $\DER^K$ too. As we already noticed, an object $k\in K$ gives rise to a 
functor $k^*:\Der K\to \Der\0$. Applying the categorical exponential law we obtain a functor
$\dia_K:\Der K\to\Der\0^K$ which sends every $X\in\Der K$ to its {\em underlying diagram} in $\Der\0^K$, where by $J^K$ we understand the category of all functors $K\to J$. Replacing $\DER$ with $\DER^J$ we get a functor 
\[\dia_{K,J}:\DER^J(K)=\Der{J\times K}\to\Der J^K.\] The derivator $\DER$ is called {\em strong} if the functor $\dia_{\one,J}$ is full
and essentially surjective for each small category $J$, (see \cite[Definition 1.8]{Gr13}).
A derivator is called {\em pointed} if its underlying category has a zero object (that is, it is pointed). Because it is quite technical, we do not recall here 
the definition of a {\em stable} derivator; we refer the interested reader to \cite[Definition 4.1]{Gr13}. 
Actually we only need the facts stated in \cite[Theorem 4.16 and Corollary 4.19]{Gr13}: 
If $\DER$ is a stable derivator, then for every small category $K$ 
the category $\Der K$ has a canonical triangulated category structure, and for every functor $u:J\to K$ the induced functors $u^*$, $u_!$ and $u_*$ are triangle functors.    

From now on let $\DER:\Cat\opp\to\CAT$ be a stable derivator.
Consider the poset $(\Z,\geq)$ viewed as a category, that is 
\[\Z=[\ldots\leftarrow -2\leftarrow -1\leftarrow 0\leftarrow 1\leftarrow 2\leftarrow\ldots].\] 
For an object $X\in\Der\Z$, let 
\[(\dagger)\ \ldots\leftarrow X_{-2}\overset{x_{-1}}\leftarrow X_{-1}\overset{x_{0}}\leftarrow X_0\overset{x_{1}}\leftarrow X_1\overset{x_{2}}\leftarrow X_2\leftarrow\ldots \]
be the underlying diagram of $X$. 
For all $n\in\Z$ consider the subcategories $\Z_{\geq n}=\{k\in\Z\mid k\geq n\}$ and $\Z_{\leq n}=\{k\in\Z\mid k\leq n\}$. 
The increasing inclusion maps 
$[\geq n]:\Z_{\geq n}\to\Z$ and $[\leq n]:\Z_{\leq n}\to\Z$ can be viewed as fully faithful functors between the respective categories. 
Therefore they induce functors depicted in the diagram: 

\[\xymatrix{ \Der{\Z_{\geq n}} \ar@/^/[rr]^{[\geq n]_!}\ar@/_/[rr]_{[\geq n]_*}
&&\Der\Z\ar[ll]|-{[\geq n]^*}\ar[rr]|-{[\leq n]^*}&&\Der{\Z_{\leq n}}\ar@/_/[ll]_{[\leq n]_!}\ar@/^/[ll]^{[\leq n]_*} }. \]

Note that $[\geq n]^*(X)_k=X_{[\geq n](k)}=X_k$  for all $k\geq n$ and $[\leq n]^*(X)_k=X_k$, for all $k\leq n$. Moreover, 
the left and right homotopy Kan extensions $[\geq n]_!$ and $[\geq n]_*$, respectively $[\leq n]_!$ and $[\leq n]_*$ are fully faithful by \cite[Proposition 1.20]{Gr13}.

By $\epsilon_{\geq n}:[\geq n]_!\cdot[\geq n]^*\to\id_{\Der\Z}$ and $\eta_{\leq n}:\id_{\Der\Z}\to[\leq n]_*\cdot[\leq n]^*$ we will denote the respective morphisms of the adjunction.

\begin{lem}\label{s-alpha}
  There exists an autoequivalence $s:\Der\Z\to\Der\Z$ and natural transformation $\alpha:\id_{\Der\Z}\Rightarrow s$ such that for all $X\in\Der\Z$ and all $k\in\Z$ we have 
  $s(X)_k=X_{k-1}$, $\alpha(X)_k$ is the map $X_k\to X_{k-1}$ in the 
 diagram  $(\dagger)$ and
 $\alpha\cdot s=s\cdot\alpha$ as natural transformations $s\Rightarrow s^2$. 
\end{lem}

\begin{proof}
 Let $[-1]:\Z\to\Z$ given by $[-1](k)=k-1$. Then $[-1]$ is an order isomorphism, therefore $[-1]^*:\Der\Z\to\Der\Z$ is an autoequivalence. 
 Denote $s=[-1]^*$. By construction $s(X)_k=X_{k-1}$ for all $k\in\Z$. Moreover $k\geq k-1$ for all $k\in\Z$ providing a (unique) natural transformation 
 $\phi:\id_\Z\Rightarrow[-1]$. If we denote by $\alpha=\phi^*:\id_{\Der\Z}\Rightarrow s$ the induced natural transformation, then 
 clearly $\alpha(X)_k$ is the respective map from $(\dagger)$.  Since $\phi\cdot[-1]=[-1]\cdot\phi$ is the unique 
 natural transformation $[-1]\Rightarrow [-1]^2=[-2]$ (with the obvious meaning for the map $[-2]:\Z\to\Z$), we deduce $\alpha\cdot s=s\cdot\alpha$.
\end{proof}

\begin{lem}\label{pointwise} For all $X\in\Der\Z$ and all $k\in\Z$, we have:
 \begin{enumerate}[{\rm (a)}]
 \item $[\geq n]_!\cdot[\geq n]^*(X)_k\cong\begin{cases} X_k \hbox{ for }k\geq n\\ X_n \hbox{ for }k<n
 \end{cases}$; 
 \item $[\leq n]_*\cdot[\leq n]^*(X)_k\cong\begin{cases} X_n \hbox{ for }k>n\\ X_k \hbox{ for }k\leq n
 \end{cases}$. 
 \end{enumerate}
\end{lem}

\begin{proof}
We only prove (a), because (b) is its dual.  For an object $X\in\Der\Z$ we have $[\geq n]^*(X)_k=X_k$ for all $k\in\Z_{\geq n}$.  
Further the left homotopy Kan extension is computed by 
\[[\geq n]_!\cdot[\geq n]^*(X)_k=\hocolim_{\Z_{\geq n}/k}\pi^*([\geq n]^*(X)),\] for all $k\in\Z$, where  
\[\Z_{\geq n}/k=\{(i,i\to k)\mid i\in\Z_{\geq n}\}\cong\begin{cases}\Z_{\geq k}\hbox{ for }k\geq n\\ \Z_{\geq n}\hbox{ for }k<n\end{cases}\] is 
the category of objects $[\geq n]$-over $k$ and $\pi:\Z_{\geq n}/k\to\Z_{\geq n}$ is the functor which forgets the morphism component. 
Thus \[[\geq n]_!\cdot[\geq n]^*(X)_k=\begin{cases}\hocolim(X_k\leftarrow X_{k+1}\leftarrow\ldots)\cong X_k \hbox{ for }k\geq n\\ 
\hocolim(X_n\leftarrow X_{n+1}\leftarrow\ldots)\cong X_n \hbox{ for }k<n\end{cases}. \]
Note that in both cases above, the diagram whose homotopy colimit is to be computed has an terminal object, hence, according to \cite[Lemma 1.19]{Gr13}, 
the homotopy colimit is isomorphic to the evaluation at this 
object. 
\end{proof}

Next we need to define those objects of $\Der\Z$ which are somehow bounded bellow or above. In order to do this, 
we refer to the underlying diagram $\dagger$ of an object $X\in\Der\Z$.  Then we define:
\[\Der\Z(\geq n)=\{X\in\Der\Z\mid x_k\hbox{ is an isomorphism for all } k\geq n\},\]
\[\Der\Z(\leq n)=\{X\in\Der\Z\mid X_k=0\hbox{ for all } k>n\}.\] 

From these definitions  and from Lemma \ref{pointwise} we can deduce immediately:

\begin{cor}\label{isos}  
 The following natural transformations:
 \begin{enumerate}[{\rm (a)}]
 \item $\epsilon_{\geq n}(X):[\geq n]_!\cdot[\geq n]^*(X)\to X$, with $X\in\Der\Z(\geq n)$; 
 \item $\eta_{\leq n}(X):X\to[\leq n]_*\cdot[\leq n]^*(X)$, with $X\in\Der\Z(\leq n-1)$ 
 \end{enumerate} are isomorphisms.
\end{cor}

The next Lemma records some immediate properties of the full subcategories $\Der\Z(\geq n)$ and $\Der\Z(\leq n)$ of $\Der\Z$: 

\begin{lem}\label{D>n-properties}
The following statements hold, for any $n\in\Z$:
\begin{enumerate}[{\rm (a)}]
\item $\Der\Z(\geq n)=\Img[\geq n]_!$ and $\Der\Z(\leq n-1)=\Ker[\geq n]^*$.
\item $\Der\Z(\geq n+1)\subseteq\Der\Z(\geq n)$ and $\Der\Z(\leq n)\subseteq\Der\Z(\leq n+1)$.  
\item The subcategories $\Der\Z(\geq n)$ and $\Der\Z(\leq n)$ of $\Der\Z$ are triangulated. 
\item The inclusion functor $\Der\Z(\geq n)\to\Der\Z$ has a right adjoint given by the assignment $Y\mapsto[\geq n]_!\cdot[\geq n]^*(Y)$ for all $Y\in\Der\Z(\geq n)$. 
\item There is a Bousfield localization sequence \[\Der\Z(\geq n)\rightleftarrows\Der\Z\rightleftarrows\Der\Z(\leq n-1)\] 
where the right adjoint of the first inclusion is the one defined in (d) above. For later use, let denote by 
$\Der\Z\to\Der\Z(\geq n),\ X\mapsto X(\geq n)$ and $\Der\Z\to\Der\Z(\leq n)$,  $X\mapsto X(\leq n)$ 
the right, respective the left adjoint of the appropriate inclusion functor.
\end{enumerate}
\end{lem}

\begin{proof}
(a) First note that $\Img[\geq n]_!$ consists exactly from those $X\in\Der\Z$ for which $\epsilon_{\geq n}(X):[\geq n]_!\cdot[\geq n]^*(X)\to X$ is an isomorphism. 
Next apply Corollary \ref{isos}(a) in order to establish the inclusion $\Der\Z(\geq n)\subseteq\Img[\geq n]_!$ and Lemma \ref{pointwise} for the converse inclusion.  

(b). These inclusions can be proven directly from the definitions of $\Der\Z(\geq n)$ and $\Der\Z(\leq n)$. 

(c). This statement follows from (a), because the functors $[\geq n]_!$ and $[\geq n]^*$ are triangulated and $[\geq n]_!$ is fully faithful (see \cite[Proposition 1.20]{Gr13}).  

(d). This follows since the fully faithful functor $[\geq n]_!$, possessing a right adjoint, induces an equivalence between $\Der{\Z_{\geq n}}$ and $\Der\Z(\geq n)=\Img[\geq n]_!$.  

(e). Because $\Der\Z(\leq n-1)=\Ker[\geq n]_!$, standard arguments concerning Bousfield localization (see \cite[Section 9.1]{Nee01}) implies that the inclusion functor 
$\Der\Z(\leq n-1)\to\Der\Z$ has a left adjoint, fitting in a Bousfield localization sequence as required.
\end{proof}

\begin{lem}\label{construct-morph} Fix $n\in\Z$ and let $X,Y\in\Der\Z(\leq n)$. If $g_k:X_k\to Y_k$, $k\in\Z$ are morphisms in $\Der\0$ such that $(g_k)_{k\in\Z}$ 
is a morphism in $\Der\0^\Z$ (that is, the underlying diagram commutes), then there is a morphism $g:X\to Y$ in $\Der\Z$ such that $\dia_\Z(g)_k=g_k$, for all $k\in\Z$.  
\end{lem}

\begin{proof}
 First use \cite[Proposition 11.3]{KN13} in order to construct a map $g'$ in $\Der{\Z_{\leq n}}$ such that $g'_k=g_k$ for all $k\leq n$. But $[\leq n]:\Z_{\leq n}\to\Z$ 
 is a cosieve in the sense of \cite[Definition 1.22]{Gr13}. Thus, by \cite[Proposition 3.6]{Gr13}, the functor 
 $[\leq n]_!:\Der{\Z_{\leq n}}\to\Der\Z$ acts by completing with zero. Thus the map $g=[\leq n]_!(g')$ has the requested properties.  
\end{proof}

\begin{lem}\label{f7} 
 With the notations above, every map $f:X\to Y$ in $\Der\Z$ fits in a $3\times3$ diagram as required in the axiom F7.  
\end{lem}

\begin{proof}
First, observe that the autoequivalence $s:\Der\Z\to\Der\Z$ and the natural transformation $\alpha:\id_{\Der\Z}\Rightarrow s$ induce a morphism of derivators 
$s:\DER^\Z\to\DER^\Z$, that is a pseudo-natural transformation (see \cite[Definition 7.5.2]{B94}) as in the definition from \cite[Section 2.1]{Gr13}), respectively
a modification (see \cite[Definition 7.5.3]{B94}) $\alpha:\id_{\DER^Z}\rightsquigarrow s$. Indeed, for every small category $I$, we define 
\[s_I=([-1]\times\id_I)^*:\DER^\Z(I)=\Der{\Z\times I}\to\Der{\Z\times I}=\DER^\Z(I)\]
and \[\alpha_I=(\phi\times\id_{\id_I})^*:\id_{\DER^\Z(I)}=\id_{\Der{\Z\times I}}\Rightarrow s_I,\] satisfying the coherence relations required 
in the definitions of the respective notions. Moreover, since $s$ is an autoequivalence, it follows that it preserves homotopy Kan extensions in the sense of 
\cite[Definition 2.2]{Gr13}. That is, for every functor $u:I\to J$ we have $s_J\cdot u_!=u_!\cdot s_I$ and $s_J\cdot u_*=u_*\cdot s_I$, where 
$u_!,u_*:\DER^\Z(I)\to\DER^\Z(J)$ are the left, respectively right adjoint of $u^*=\DER^\Z(u)$. We claim that for every $X\in\DER^\Z(I)$ we have 
\[\alpha_J(u_!(X))=u_!(\alpha_I(X))\hbox{ and }\alpha_J(u_*(X))=u_*(\alpha_I(X)).\] We focus on the first equality, the argument for the second being similar.
For $k\in\Z$ consider the morphism of derivators $k^*:\DER^\Z\to\DER$ and for an $X\in\DER^\Z(I)$ denote  $Y=u_!(X)\in\DER^\Z(J)$. 
If by $X_{k_I}$ we denote the image of $X$ under the functor $k^*_I:\DER^\Z(I)\to\Der I$, then $\dia_{\Z,I}(X)$ and $\dia_{\Z,I}(Y)$ are the diagrams: 
\[\ldots\leftarrow X_{-1_I}\overset{x_{0_I}}\leftarrow X_{0_I}\overset{x_{1_I}}\leftarrow X_{1_I}\leftarrow\ldots,\] respectively 
\[\ldots\leftarrow Y_{-1_J}\overset{y_{0_J}}\leftarrow Y_{0_J}\overset{y_{1_J}}\leftarrow Y_{1_J}\leftarrow\ldots.\]  By definition of $\alpha$ we have 
$\alpha_J(Y)_{k_J}=y_{k_J}$ and $\alpha_I(X)_{k_I}=x_{k_I}$. According to \cite[Proposition 2.5]{Gr13}, the $k^*$ preserves homotopy Kan extensions, 
hence  we have $Y_{k_J}=u_!(X)_{k_J}=u_!(X_{k_I})$, so $\alpha_J(u_!(X))=\alpha_J(Y)=u_!(\alpha_{I}(X))$.

Recall that the triangulated structure in $\Der I$, were $I$ is an arbitrary small category, is constructed as follows: 
Let $K=\two\times\one\setminus\{(1,1),(2,1)\}$, $i_0:\one\to K$, $i_0(x)=(x,0)$ and $i_1:K\to\two\times\one$ be the inclusion functor. 
Denote by $T:\DER^I(\one)\to\DER^I({\two\times\one})$ the composition
$i_{1!}\cdot i_{0*}$. A (distinguished) triangle in $\Der I$ is a sequence
\[T(f)_{(0,0)}\to T(f)_{(1,0)}\to T(f)_{(1,1)}\to T(f)_{(2,1)}\]for some $f\in\DER^I(\one)$ (see \cite[Theorem 4.16]{Gr13}).  

Now let $g:Y'\to Y$ be a map in $\Der\Z$ and complete it to a triangle $Y'\overset{g}\to Y\to Y''\overset{+}\to.$ The square 
\[\diagram Y'\rto^{g}\dto_{\alpha(Y')} & Y\dto^{\alpha(Y)} \\	
 s(Y')\rto_{s(g)} & s(Y)
\enddiagram\]
can be viewed as a morphism in $\DER^\Z(\one)$ between $\alpha(Y')$ and $\alpha(Y)$, or equivalently as an object in $\DER^\Z(\one)^\one$. Since, together with $\DER$,  
the derivator $\DER^\Z$ is strong too, we find an object $\tilde g$ in $\DER^\Z(\one\times\one)=\DER^{\Z\times\one}(\one)$ whose underlying diagram is the square above. Using the claim 
showed in the first part of this proof, we know that the underlying diagram of the $T(\tilde g)\in\DER^{\Z\times\one}(\two\times\one)$ is 
\[\mbox{\diagram 
& Y'\rrto^{g}\ddto\dlto_{\alpha(Y')} & & Y\rrto\ddto\dlto_{\alpha(Y)}& &0\ddto\dlto \\
s(Y')\ddto\rrto^{\ \ \ \ \ \ s(g)}& &s(Y)\rrto\ddto & & 0\ddto& \\
& 0\rrto\dlto& &Y''\rrto\dlto_{\alpha(Y'')} & &\Sigma Y'\dlto_{\alpha(\Sigma Y')} \\
0\rrto& &s(Y'')\rrto & &\Sigma s(Y') &
\enddiagram}\]

Fix now $f:X\to Y$ in $\Der\Z$ as in the hypothesis. Because of the uniqueness of the morphisms $f(\geq1)$ and $f(\leq0)$ the diagram 
\[\mbox{\tiny\diagram 
& X(\geq 1)\rrto\ddto\dlto_{f(\geq1)} & & X\rrto\ddto\dlto_{f}& &0\ddto\dlto\\
Y(\geq1)\ddto\rrto& &Y\rrto\ddto & & 0\ddto&\\
& 0\rrto\dlto& &X(\leq0)\rrto\dlto_{f(\leq0)} & &\Sigma X(\geq1)\dlto_{\Sigma f(\geq1)}\\
0\rrto& &Y(\leq0)\rrto & &\Sigma Y(\geq1) &
\enddiagram}\]
is the underlying diagram of $T(\tilde f)\in\DER^{\Z\times\one}(\two\times\one)$, where $\tilde f$ in $\DER^{\Z\times\one}(\one)$ is an object whose underling diagram is 
the first upper square above. Now we look at the last two diagrams. For both the left and the right cube, are bicartesian squares in the sense of \cite[Definition 4.1]{Gr13}, in 
$\Der{\Z\times\one}=\DER^\Z(\one)$. If, in particular, we take the map $g$ above to be $Y(\geq 1)\to Y$ then we can glue the two diagrams and 
\cite[Proposition 3.13]{Gr13} tells us that 
both the left and the right part of the obtained diagram are bicartesian squares in $\DER^\Z(\one)$. Therefore the whole diagram obtained after gluing, 
produces a triangle: \[\alpha(Y(\geq1)\cdot f(\geq1)\to\alpha(Y)\cdot f\to\alpha(Y(\leq0))\cdot f(\leq0)\overset{+}\to\] in 
$\DER^\Z(\one)$.
Finally the cone functor ${\mathsf C}:\DER^\Z(\one)\to\DER^\Z(\0)=\Der\Z$ defined as in \cite[Definition 3.18(3)]{Gr13} is triangulated. Applying it to the above triangle in  
$\DER^\Z(\one)$ we obtain the $3\times3$ diagram whose rows and columns are triangles in $\Der\Z$ exactly as it is required in F7.
\end{proof}

Remark that the proof of the above Lemma suggests that middling--good morphisms of triangles in $\Der\0$
are actually diagrams representing  triangles in $\Der\one$. 

\begin{thm}\label{main-thm}
For every stable derivator $\DER$, the underlying triangulated category $\Der\0$  has an f-enhancement, satisfying the additional axiom F7. 
\end{thm}

\begin{proof} Let 
\[\X=\bigcup_{m\leq n}\left(\Der\Z(\geq m)\cap\Der\Z(\leq n)\right).\] According to Lemma \ref{D>n-properties}(b), $\Der\Z(\geq m)$ and $\Der\Z(\leq n)$ 
are triangulated subcategories of $\Der\Z$, hence the same is true for the intersection. Further $\X$ is triangulated 
as a directed union of triangulated subcategories. Note that an object $X\in\X$ has the underling diagram of the form:
\[(\ddagger)\ \ldots\underset{\cong}\leftarrow X_{m-1} \underset{\cong}{\overset{x_m}\leftarrow} X_{m}\leftarrow X_{m+1}\leftarrow 
\ldots \leftarrow X_{n-1}\overset{x_n}\leftarrow X_{n}\leftarrow 0\leftarrow 0\leftarrow\ldots \]

We denote 
$\X(\geq n)=\X\cap\Der\Z(\geq n)$ and $\X(\leq n)=\X\cap\Der\Z(\leq n)$, for all $n\in\Z$.  Since for all $k\in\Z$ we have $s(X)_k=X_{k-1}$, it follows  
$s(X)\in\X$ for every $X\in\X$, hence $s$ induced a well-defined functor (denoted with the same symbol) $s:\X\to X$. 
 The restriction the natural morphism $\alpha$ gives a natural morphism $\alpha:\id_\X\Rightarrow s$.  
 We obviously have $s\X(\geq n)=\X(\geq n+1)$ and $s\X(\leq n)=\X(\leq n+1)$.
 
 The functor $0:\0\to\Z$ is fully faithful, therefore \cite[Proposition 1.20]{Gr13} implies $0_!:\Der\0\to\Der\Z$ has the same property. 
 Moreover for all $T\in\Der\0$ and $k\in\Z$ we have $0_!(T)_k=\hocolim_{\0/k}\pi^*(T)$, where  
\[\0/k=\{(0,0\to k)\}\cong\begin{cases}0\hbox{ for }k\leq 0\\ \emptyset\hbox{ for }k>0\end{cases}\] is 
the category of objects $\0$-over $k$ and $\pi:\0/k\to e$ is the functor which forgets the morphism component. Therefore 
\[0_!(T)_k\cong\begin{cases}T\hbox{, for }k\leq 0\\ 0\hbox{, for }k>0\end{cases},\]
thus $\Img0_1=\X(\geq0)\cap\X(\leq0)$ and $0_!$ induces an equivalence between $\Der\0$ and $\X(\geq0)\cap\X(\leq0)$.

 We claim that $(\X, \X(\geq 0),\X(\leq 0), s,\alpha)$ is an f-enhancement for $\Der\0$, hence we have to verify the axioms F1-F6. 
 
 The axiom F1 follows by Lemma \ref{D>n-properties}(a), the axiom F5 follows by Lemma \ref{s-alpha} and the axiom F2 is direct a consequence of the construction of $\X$. 
 
 The axioms F3 and F4 are equivalent to the condition that \[\X(\geq n)\rightleftarrows\X\rightleftarrows\X(\leq n-1)\] 
 is a Bousfield localization sequence.  
 In order to prove this it is enough to observe that the functors from the Bousfield localization sequence of Lemma \ref{D>n-properties}(e) 
 restrict well between subcategories involved here.  

For proving F6 let $X\in\X(\geq 1)$ and $Y\in\X(\leq 0)$, that is $X_k\cong X_1$ for all $k\leq 1$ and $Y_k=0$ for all $k>0$. We want to show that the map 
\[\varphi:\X(s(Y),X)\to\X(Y,X)\hbox{ given by }\varphi(g)=g\cdot\alpha(Y)\] is bijective. 
If $\varphi(g)=0$ for some $g:s(Y)\to X$, then $g\cdot\alpha(Y)=0$ so $(g\cdot\alpha(Y))_k=0$ for all $k\in\Z$. Since for $k\leq 1$ the map $x_k:X_{k}\overset{\cong}\longrightarrow X_{k-1}$ is an isomorphism, 
the construction of $\alpha$, 
implies $g_k=(g\cdot\alpha(Y))_{k-1}=0$. 
On the other side, for $k>1$ we have $g_k=0$ because $s(Y)_k=Y_{k-1}=0$. Therefore all $g_k$ vanish. Observe that $g$ is an object in $\Der\Z^\one$, 
hence by the strongness assumption on $\DER$, there is an object $\tilde g\in\Der{\Z\times\one}$ such that 
$\dia_{\one,\Z}(\tilde g)=g$. We deduce $(\tilde g)_k=g_k=0$ 
for all $k\in\Z$, therefore the axiom D2 tells us that $\tilde g=0$, and   
consequently $g=0$. This proves that $\varphi$ is injective. For showing that 
$\varphi$ is surjective, let $f\in\X(Y,X)$. Then  take 
\[g_k=\begin{cases} x_k^{-1}f_{k-1}:s(Y)_k=Y_{k-1}\to X_{k-1}\overset{\cong}\to X_k\hbox{, for }k\leq 1\\  0:s(Y)_k=0\to X_k\hbox{, for }k>1\end{cases}.\] 
In this way we constructed a morphism 
$(g_k)_{k\in\Z}$ in $\Der\0^\Z$. According to Lemma \ref{construct-morph}, we find a morphism $g:s(Y)\to X$ in $\Der\Z$, or equivalently in $\X$, 
such that $g\cdot\alpha(Y)_k=f_k$ for all $k\in\Z$. Noting that our assumptions imply that the categories $\Der J$ and the functors $u^*, u_*, u_!$ are additive for all  functors
$u:J\to K$ in $\Cat$ (see \cite[Corollary 4.14]{Gr13}), 
it follows  $(g\cdot\alpha(Y)-f)_k=0$, and the same argument as before shows that from the strongness and D2 we deduce 
$g\cdot\alpha(Y)=f$. 

Finally, the axiom F7 follows by Lemma \ref{f7}. Again we can observe that if we start with a map $f:X\to Y$ in $\X$ rather than in $\Der\Z$, then the 
whole $3\times3$ diagram constructed in this Lemma lies in $\X$. 
 \end{proof}

\begin{expl} Let $\A$ be a sufficiently nice abelian category (e.g. a Grothendieck category), as in \cite[Example 1.2 (2)]{Gr13}. 
Then the prederivator $\DER:\Cat\opp\to\CAT$, given by 
$\Der I=\D(\A^I)$, where $\D(\A^I)$ denotes the derived category of the abelian category $\A^I$, is actually a strong stable derivator. 
Recall that $\D(\A^I)$ is constructed as follows: The objects of the category of complexes $\C({\A^I})$ are
\[\ldots\to X^{n-1}\overset{d^{n-1}}\longrightarrow X^{n}\overset{d^n}\longrightarrow X^{n+1}\to\ldots\] with $X^n\in\A^I$, for all $n\in\Z$ and $d^n\cdot d^{n-1}=0$. The morphisms $d^n$, $n\in\Z$ are called {\em differentials}. Morphisms in $\C(\A^I)$ are collection of morphisms in $\A^I$ commuting with the differentials. 
Then we get a functor 
\[\Ho^n:\C(\A^Z)\to\A^\Z\hbox{ given by}\Ho^n(X)=\Ker d^n/\Img d^{n-1}\] 
called the $n$-th {\em cohomology} of $X$. 
Then $\D(\A^I)$ is obtained from $\C(\A^I)$ by formally inverting {\em quasi-isomorphisms}, that is those morphisms $f:X\to Y$ in $\C(\A^I)$ which 
induce isomorphisms in cohomology. 

On the other hand, note that objects in $\A^\Z$ are diagrams of the form 
(recall that $\Z$ is regarded as the category $(\Z,\geq)$):
\[X=\ldots\leftarrow X_{-1}\leftarrow X_0\leftarrow X_1\leftarrow\ldots,\] with $X_k\in\A$. 
Therefore an object in $X\in\C(\A^\Z)$ is a diagram of the form: 
\[X=\left(\diagram 
X^{n-1}_{k-1}\dto_{d^{n-1}_{k-1}} & X^{n-1}_k\dto^{d^{n-1}_k}\lto\\
X^{n}_{k-1} & X^n_k\lto\enddiagram\right)_{n,k\in\Z}\]
Then the f-enhancement for $\D(\A)$ constructed in Theorem \ref{main-thm} above  looks as follows: The triangulated category $\X$ is the subcategory of 
$\D(\A^Z)$ consisting of those objects for which the  
horizontal morphisms are isomorphisms for sufficiently small 
$k$ and $X^n_k=0$ for sufficiently large $k$. The autoequivalence $s$ acts on $X\in\D(\A^\Z)$ by shifting the columns, 
that 
is the initial column $k-1$ becomes the column $k$ in $s(X)$, and $\alpha$ is the obvious natural transformation. 
Note that the unique difference between our setting and the 
filtered derived category of \cite[Example A 2]{Be87} is 
the fact that we don't insist that the horizontal morphisms are the image in the derived category of monomorphisms in the category of complexes. However, checking carefully all arguments in \cite[Proposition 6.3]{Sch11}, one can observe 
that this requirement is superfluous.  
\end{expl}

We end this paper with an immediate application where we record a new proof for an already known result, 
namely we show how can be used the realization functor in 
order to compute the map spaces between Eilenberg--Mac Lane spectra. 
Recall that there are more construction of the category of spectra leading to the same, up to an equivalence, stable homotopy category. 
We follow here \cite[Section 3]{HSS99}. According to this approach, the category $\HoSp$ is (equivalent to) the homotopy category of stable model category, 
namely the model category $\Sp$ of symmetric spectra. Therefore $\HoSp=\mathrm{Ho}(\Sp)$, where by $\mathrm{Ho}$ we denote the homotopy category of a model category
, see \cite[Section 3.2]{HSS99}). For any abelian group $A$ we denote by $(A,n)$ the Eilenberg--Mac Lane spectrum with homology 
$A$ concentrated in degree $n$ (for details, see \cite{EML53}). 

\begin{prop}\label{spectra} With the notations above, we have:  \[\HoSp((A,n),(B,m))=\begin{cases} 0\hbox{ if }m<n\\
		    \Hom_\Z(A,B)\hbox{ if }m=n\\
                     \Ext_\Z(A,B)\hbox{ if }m=n+1
                    \end{cases}\]
\end{prop}

\begin{proof}  Since the model category $\Sp$ is combinatorial, the associated derivator 
 \[\DER_\Sp:\Cat\opp\to\CAT,\ \DER_\Sp(I)=\mathrm{Ho}(\Sp^I),\]
is stable, by \cite[Example 4.2(1)]{Gr13}. 
Clearly $\HoSp=\DER_\Sp(\0)$, therefore it has an f-enhancement, by Theorem \ref{main-thm} above. Let $\pi_i(X)$ be the i-th homotopy group 
of an object $X\in\HoSp$ and, for shifting to our cohomological grading, put $\pi^{i}(X)=\pi_{-i}(X)$. Then define
\[\HoSp^{\leq0}=\{X\in\HoSp\mid \pi^i(X)=0\hbox{ for all }i>0\}\] and 
\[\HoSp^{\geq0}=\{X\in\HoSp\mid \pi^i(X)=0\hbox{ for all }i<0\}.\]  Then $(\HoSp^{\leq0},\HoSp^{\geq0})$ 
is a t-structure in $\HoSp$, called the {\em Postnikov t-structure} whose heart is the category of Eilenberg--Mac Lane spectra which is equivalent to 
the category $\Ab$ of abelian groups (see \cite[Examples 2.3 and 2.16]{FI07}). 
Note that this already implies $\HoSp((A,n),(B,n))=\Hom_\Z(A,B)$.
By \cite[Theorem 3.11]{PV17} there is a realization functor 
$\real:\D^b(\Z)\to\HoSp$ which coincide with the inclusion on the heart, and induces isomorphisms \[\D^b(\Z)(A,\Sigma^nB)\to\HoSp(A,\Sigma^nB)\]
for all $A,B\in\Ab$, and all $n\leq 1$. Now the conclusion follows from the description of morphism sets the derived category.
\end{proof}

\begin{rem}(a). One can note the following properties of the functor \[\real:\D^b(\Z)\to\HoSp\] defined in the proof of  Corollary \ref{spectra} above (for unexplained terms see \cite{PV17}):
\begin{enumerate}[(1)]
\item It is a triangulated functor, because the f-enhancement constructed in Theorem \ref{main-thm} satisfies F7, so \cite[Theorem A.3]{PV17} applies.  
 \item It is not fully faithful, because 
$\HoSp((A,n),(B,m))$ is not necessary zero for $m>n+1$.
\item It is t-exact with respect to standard t-structure in $\D^b(\Z)$ and the Postnikov t-structure in $\HoSp$. 
\item $\Ho^i=\Ho^i_\HoSp\cdot\real$ for all $i\in\Z$, where $\Ho^i$ and $\Ho_\HoSp^i$ are the cohomology functors associated with the standard t-structure in
$\D^b(\Z)$, respectively the Postnikov t-structure in $\HoSp$. 
\item The functor $\real$ makes the following diagram commutative: 
\[\diagram \C^b(\Z)\rrto^{\sim}\dto && \mathcal{H}\dto \\
	  \D^b(\Z)\rrto^{\real} && \HoSp
\enddiagram\]
where $\C^b(\Z)\to\D^b(\Z)$ is the canonical quotient functor from the bounded category of complexes of abelian groups $\C^b(\Z)$ to the bounded derived category, 
and the left part of the diagram is constructed as follows: 
The Postnikov t-structure in $\HoSp$ induces a t-structure in the f-enhancement $\widetilde\HoSp$ of $\HoSp$, whose heart is denoted by $\mathcal H$ and it is equivalent to 
$\C^b(\Z)$. The functor $\mathcal{H}\to\HoSp$ is the restriction of a canonically constructed functor $\omega:\widetilde\HoSp\to\HoSp$ (for details, see \cite[Section 3]{PV17}).

\end{enumerate}
\end{rem}

\end{document}